\documentclass[runningheads,envcountsame]{llncs}

\title{On the structure of balanced residuated \\ partially ordered monoids}

\author{Stefano Bonzio\inst{1}\orcidID{0000-0002-5959-5868}\and Jos\'e Gil-F\'erez\inst{2}\orcidID{0000-0002-3086-6070} \and \\
Peter Jipsen\inst{2}\orcidID{0000-0001-8608-808X} \and
Adam P\v{r}enosil \inst{3}\orcidID{0000-0003-0377-0783} \and \\
Melissa Sugimoto\inst{2}\orcidID{0009-0004-7168-8954}}

\authorrunning{S. Bonzio, J. Gil-F\'erez, P. Jipsen, A. P\v{r}enosil, M. Sugimoto}

\institute{
University of Cagliari, Italy\\
\email{stefano.bonzio@unica.it}
\and
Chapman University, Orange CA 92866, USA\\
\email{\{gilferez,jipsen,msugimoto\}@chapman.edu}
\and
University of Barcelona, Spain\\
\email{adam.prenosil@gmail.com}
}

\usepackage[T1]{fontenc}
\usepackage{amsmath,amssymb, mathtools}
\usepackage{thm-restate}
\usepackage{array}
\usepackage{color}
\usepackage[shortlabels]{enumitem}
\usepackage[colorlinks=true,allcolors=blue]{hyperref} %hidelinks
\usepackage{leftindex}
\usepackage[kerning]{microtype}
\usepackage{tikz}
\usepackage{enumitem}

\makeatletter
\newcounter{dummy}
\newcommand\myitem[1][]{\item[#1]\refstepcounter{dummy}\def\@currentlabel{#1}}
\makeatother

\tikzstyle{every picture} = [scale=.45]
\tikzstyle{every node} = [draw, fill=white, circle, inner sep=0pt, minimum size=4pt]
\tikzstyle{d} = [very thick]
\tikzstyle{i} = [draw, fill=black, circle, inner sep=0pt, minimum size=4pt]
\tikzstyle{n} = [draw=none, rectangle, inner sep=2pt]

\let\:\colon

\let\ld\backslash
\let\le\leqslant
\let\meta\looparrowright

\let\rd/
\let\setminus\smallsetminus
\let\subset\subseteq
\let\sle\sqsubseteq

\newcommand\lscr[2]{\leftindex_{#1}{#2}}
\newcommand{\m}{\mathbf}
\newcommand\nbd[1]{\protect\nobreakdash#1\hspace{0pt}}
\newcommand\nein{\mathord{\sim}}
\newcommand\no{\mathord{-}}

\newcommand\pair[1]{\langle#1\rangle}

\DeclareMathOperator\Idp{Id^+\!}
\newcommand\category[1]{\expandafter\newcommand\csname#1\endcsname{\textsf{#1}}}
\category{Meta}
\category{Alg}

\begin{document}

\maketitle              

\begin{abstract}
A \emph{residuated poset} is a structure $\pair{A,\le,\cdot,\ld,\rd,1}$ where $\pair{A,\le}$ is a poset and $\pair{A,\cdot,1}$ is a monoid such that the residuation law $x\cdot y\le z\iff x\le z/y\iff y\le x\backslash z$ holds. A residuated poset is \emph{balanced} if it satisfies the identity $x\backslash x \approx x/x$. By generalizing the well-known construction of P\l{}onka sums, we show that a specific class of balanced residuated posets can be decomposed into such a sum indexed by the set of positive idempotent elements. Conversely, given a semilattice directed system of residuated posets equipped with two families of maps (instead of one, as in the usual case), we construct a residuated poset based on the disjoint union of their domains. We apply this approach to provide a structural description of some varieties of residuated lattices and relation algebras.
\end{abstract}

\section{Introduction}

A \emph{residuated partially ordered monoid}, or \emph{residuated poset} for short, is a structure of the form $\m A = \pair{A,\le,\cdot,\ld,\rd,1}$ such that $\pair{A,\le}$ is a poset, $\pair{A,\cdot,1}$ is a monoid, and the \emph{multiplication} $\cdot$ has \emph{residuals} $\ld$ and $\rd$, that is,
\[
x\cdot y \le z \quad\iff\quad x \le z\rd y \quad\iff\quad y \le x\ld z.
\]
As usual, we denote the product $x\cdot y$ simply by $xy$. The terminology ``residuated posets'' emphasizes the similarity with residuated lattices since these algebras are residuated posets in which the partial order is in fact a lattice. Residuated lattices form a variety, whereas residuated posets are a po-variety, i.e. they are defined by inequations and all fundamental operations are either order-preserving or order-reversing in each argument~\cite{Pi04}.

The class of residuated posets includes all groups (ordered by the antichain order), partially ordered groups, hoops, Brouwerian semilattices and generalized Boolean algebras, as well as all subreducts of residuated lattices. In the case of groups with the antichain order, the residuals are $x\ld y=x^{-1}y$ and $x/y=xy^{-1}$. Under this interpretation of the residuals, groups satisfy the identity $x\ld x=x/x$, which can fail in relation algebras and complex algebras of groups. Residuated posets in which this identity holds are called \emph{balanced}, and they are the main focus of our investigations. 

P\l{}onka sums are reviewed in Section~\ref{plonka}, and they can be an effective tool for decomposing algebras into a semilattice sum of simpler components. Here we generalize P\l{}onka sums to po-algebras, and in particular to residuated posets. Previously, P\l{}onka sums have been used for a structural description of locally integral involutive po-monoids in~\cite{GFJiLo23}. These are a subclass of residuated posets in which the residuals are definable from two linear negations $\nein,\no$ via $x\ld y=\nein(\no y\cdot x)$ and $x\rd y=\no(y\cdot \nein x)$. Local integrality is a property that ensures the components of our decompositions are integral, i.e., have the monoid identity as top element. In the current paper we show how to obtain similar structural results for a much larger class of balanced residuated posets that are neither involutive nor locally integral.

\section{Background}

\subsection{P\l{}onka Sums and Partition Functions}\label{plonka}

Our main result relies on a generalization of the notion of a P\l{}onka sum. This construction was first introduced and studied by J. P\l{}onka in~\cite{Plonka67,Plonka68a,Plonka68b} (for more recent expositions see~\cite{PR92} and~\cite{BPP22}). Given a join semilattice of \emph{indices} $\m I = (I, \lor)$, a family of homomorphisms $\Phi = \{\varphi_{ij}\: \m A_i\to \m A_j : i\le j \text{ in } \m I\}$ between algebras of the same type is said to be a \emph{semilattice directed system} if $\varphi_{ii}$ is the identity on~$\m A_i$ for every index $i$ and $\varphi_{jk}\circ\varphi_{ij} = \varphi_{ik}$, for all indices $i \le j \le k$. If the algebras contain constants, we assume that $\m I$ has a least element $\bot$. The \emph{P\l{}onka sum} of a semilattice directed system $\Phi$ is an algebra $\m S$ of the same type defined on the disjoint union of the universes $S = \biguplus_{i\in I} A_i$. For every $n$-ary operation symbol $\sigma$ and elements $a_1\in A_{i_1}$, \dots, $a_n\in A_{i_n}$,
\[
\sigma^{\m S}(a_1,\dots,a_n) = \sigma^{\m A_j}(\varphi_{i_1j}(a_1),\dots,\varphi_{i_nj}(a_n)),
\]
where $j = i_1\lor \dots\lor i_n$, and for every constant symbol $\omega$, $\omega^{\m S} = \omega^{\m A_\bot}$.

One can readily prove that the P\l{}onka sum of a semilattice directed system is well defined and it satisfies all regular equations that hold in all the algebras of the family. An equation $t_1\approx t_2$ is \emph{regular} if the variables of $t_1$ and $t_2$ are the same. P\l{}onka sums can equivalently be understood in terms of so-called partition functions. A \emph{partition function} on an algebra ${\m A}$ is a binary operation $\odot\: A^2\to A$ satisfying the following conditions,\footnote{In the literature, one can find different definitions of a partition function. We are here opting for the definition that appears in~\cite{BPP22} and~\cite{PR92}, which makes use of the minimal number of conditions.} for every $n$-ary operation symbol~$\sigma$, every constant symbol~$\omega$, and all $a, b, c, a_1,\dots, a_n\in A$,
\begin{enumerate}[label=(PF\arabic*), leftmargin=*]
\item\label{PF1} $a\odot a = a$,
\item\label{PF2} $a\odot (b\odot c) = (a\odot b) \odot c $,
\item\label{PF3} $a\odot (b\odot c) = a\odot (c\odot b)$,
\item\label{PF4} $\sigma^{\mathbf{A}}(a_1,\dots,a_n)\odot b = \sigma^{\mathbf{A}}(a_1\odot b,\dots, a_n\odot b)$,
\item\label{PF5} $b\odot \sigma^{\mathbf{A}}(a_1,\dots,a_n) = b\odot a_{1}\odot \dots\odot a_n $,
\item\label{PF6} $b\odot \omega^{\mathbf{A}} = b$. 
\end{enumerate}

An operation $\odot$ satisfying~\ref{PF1}--\ref{PF3} is known as a \emph{left normal band}. The connection between partition functions and P\l{}onka sums is made clear by the following P\l{}onka's decomposition theorem.

\begin{theorem}\textnormal{\cite[Thm.~II]{Plonka67}}\label{thm:Plonka:Theorem}
Let $\mathbf{A}$ be an algebra with a partition function $\odot$.
\begin{enumerate}[(1)]
\item There exists a partition $\{ A_i : i \in I\}$ of $A$, such that any two elements $a, b \in A$ belong to the same equivalence class exactly when
\[
a\odot b = a \quad\text{and}\quad b\odot a = b.
\]
\item The relation $\le$ on $I$ defined by the condition
\[
i \le j \iff \text{there exist }a \in A_i \text{ and } b \in A_j \text{ such that } b\odot a =b
\]
is a partial order and $\m I = \pair{I,\le}$ is a join semilattice. If $\m A$ contains some constant, then $\m I$ has a least element $\bot$. 

\item Each equivalence class $A_i$ is the universe of an algebra $\m A_i$ of the same type, whose constant-free reduct is a subalgebra of the constant-free reduct of~$\m A$. If $\m A$ contains some constant, then $\m A_\bot$ is a subalgebra of $\m A$.

\item For all $i,j\in I$ such that $i\le j$ there is a homomorphism $\varphi_{ij}\: \m A_i\to \m A_j$, defined by $\varphi_{ij}(a)= a\odot b$, for any $b\in A_j$.

\item $\m A $ is the P\l{}onka sum of the semilattice directed system $\{\varphi_{ij} : i \le j \text{ in }\m I\}$.
\end{enumerate}
\end{theorem}

\subsection{Positive Idempotents in Residuated Posets}

An element $p$ of a residuated poset $\m A$ is \emph{positive} if $1\le p$, and it is \emph{idempotent} if $p^2 = p\cdot p = p$. We denote by $\Idp\m A$ the set of all positive idempotents of $\m A$. For every element $p\in\Idp\m A$, we can define the following sets:
\begin{align*}
A\rd p &= \{a\rd p : a\in A\}, & Ap &= \{ap : a\in A\}, & A_p &= \{a\in A : p = a\ld a\},\\
p\ld A &= \{p\ld a : a\in A\}, & pA &= \{pa : a\in A\}, & \lscr pA &= \{a\in A : p = a\rd a\}.
\end{align*}

\begin{lemma}\label{lem:char:A/p=Ap}
We have the following equalities, for every $p\in\Idp\m A$.
\begin{align*}
1.\quad    A\rd p &= \{a\in A : a = a\rd p\} = \{a\in A : a\le a\rd p\} = \{a\in A : ap\le a\}\\
 &= \{a\in A : ap = a\} = Ap.\\
2.\quad    p\ld A &= \{a\in A : a = p\ld a\} = \{a\in A : a\le p\ld a\} = \{a\in A : pa\le a\}\\
 &= \{a\in A : pa = a\} = pA.
\end{align*}
\end{lemma}

\begin{proof}
We will only prove part~1, since the proof of part~2 is completely analogous. First of all, notice that the three middle equalities are consequences of residuation and the fact that $p$ is positive and therefore $a\rd p \le a\rd 1 = a$ and $a = a\cdot 1 \le ap$. Now, for every $b\in A$, we have that $b\rd p = b\rd p^2 = (b\rd p)\rd p$, which shows that $A\rd p \subseteq \{a\in A : a = a\rd p\}$. The reverse inclusion is trivial. We also have that for every $b\in A$, the equality $bpp = bp$ implies that $Ap\subset \{a\in A : ap =a\}$. The reverse inclusion is also trivial.\qed
\end{proof}

\begin{lemma}\label{lem:IdpA:identities}
Let $\m A$ be a residuated poset and $a\in A$.
\begin{enumerate}[(1)]
    \item $a\rd a$ is the largest $p\in\Idp\m A$ such that $a\in p\ld A$.
    \item $a\ld a$ is the largest $p\in\Idp\m A$ such that $a\in A\rd p$.
\end{enumerate}
Moreover,
\[
\Idp\m A = \{a\rd a : a\in A\} = \{a\ld a : a\in A\},
\]
and hence the sets $\{\lscr pA : p\in\Idp\m A\}$ and $\{A_p : p\in\Idp\m A\}$ are partitions of~$A$.
\end{lemma}

\begin{proof}
Notice that $(a\rd a)a \le a$ and $1\le a\rd a$ hold by residuation, and since $(a\rd a)(a\rd a)a \le (a\rd a)a \le a$, we also have that $(a\rd a)(a\rd a) \le a\rd a$. And, by order-preservation of the product, $a\rd a = (a\rd a)\cdot 1 \le (a\rd a)(a\rd a)$. Hence, $p = a\rd a \in\Idp\m A$ and $pa \le a$, that is, $a\in p\ld A$. Let $q\in \Idp\m A$ such that $a\in q\ld A$. Then $qa\le a$ and, by residuation, $q\le a\rd a$. The proof of the second claim is analogous.

The inclusion $\{a\rd a : a\in A\}\subset\Idp\m A$ is true by part~1. For the reverse inclusion, notice that if $p\in\Idp\m A$ then $pp=p$, and therefore $p\in A\rd p$. Thus, $p=p\rd p$, by Lemma~\ref{lem:char:A/p=Ap}. The other equality is shown analogously using~2.\qed
\end{proof}

The following lemma is an immediate consequence of this.

\begin{lemma}\label{char:interdefinability}
For every $p\in\Idp\m A$, we have the following equalities.
\[
A\rd p = \smashoperator{\bigcup\limits_{p\le q\in\Idp\m A}} A_q,
\quad A_p = (A\rd p)\setminus \smashoperator{\bigcup\limits_{p < q\in\Idp\m A}} A_q,
\quad p\ld A = \smashoperator{\bigcup\limits_{p\le q\in\Idp\m A}} \lscr qA,
\quad \lscr pA = (p\ld A)\setminus \smashoperator{\bigcup\limits_{p < q\in\Idp\m A}} \lscr qA.
\]
\end{lemma}

\begin{lemma}\label{lem:Ap:closure:meets:joins}
Let $\m A$ be a residuated poset and $p\in\Idp\m A$.
\begin{enumerate}[(1)]
\item The maps $a \mapsto pa$ and $a \mapsto ap$ are join-preserving closure operators whose images are $pA$ and $A p$, respectively.
\item The maps $a \mapsto p\ld a$ and $a \mapsto a\rd p$ are meet-preserving interior operators whose images are $p\ld A$ and $A\rd p$, respectively.
\item The sets $pA = p\ld A$ and $Ap = A\rd p$ are closed under existing meets and joins.
\end{enumerate}
\end{lemma}

\begin{proof}
Consider the map $\gamma\: a\mapsto pa$. The positivity and idempotence of $p$ imply that $a \le \gamma(a)$ and $\gamma(\gamma(a)) = \gamma(a)$ for all $a\in A$, respectively. And if $a\le b$ then $\gamma(a) \le \gamma(b)$, by the monotonicity of the product. Hence, $\gamma$ is a closure operator. Its image is $\{\gamma(a) : a\in A\} = pA$. One can readily prove that if the join $\bigvee X$ exists for some $X\subset A$, then $\bigvee\{pa : a\in X\}$ also exists and $p\cdot\bigvee X = \bigvee\{pa : a\in X\}$. That is, $\gamma(\bigvee X) = \bigvee\gamma[X]$. The proof of the other claim is analogous.

Consider now the map $\delta\: a\mapsto p\ld a$. The positivity and idempotence of $p$ imply now that $\delta(a)\le a$ and $\delta(\delta(a)) = \delta(a)$, and the monotonicity in the numerator of the left residual implies the monotonicity of $\delta$. Hence, $\delta$ is an interior operator. Again, one can show that if $\bigwedge X$ exists for some $X\subset A$, then $\bigwedge\{p\ld a : a\in X\}$ also exists and $p\ld\bigwedge X = \bigwedge\{p\ld a : a\in X\}$. That is, $\delta(\bigwedge X) = \bigwedge\delta[X]$. The proof of the other claim is analogous.

Recall that the images of closure and interior operators are closed under existing meets and existing joins, respectively. Hence, since $pA = p\ld A$, by Lemma~\ref{lem:char:A/p=Ap}, this set is closed under existing meets and joins, by the previous parts.\qed
\end{proof}

\section{Decompositions of Balanced Residuated Posets}

In what follows, we will denote the term $x\rd x$ by $1_x$. Notice that, as we argued in the proof of Lemma~\ref{lem:IdpA:identities}, in every residuated poset $\m A$, $1_p = p$ for every $p\in\Idp\m A$, namely, $p\rd p = p$. Hence, $1_{1_x}\approx 1_x$ holds in every residuated poset, as well as $1_x\cdot x\approx x$. The term $x\ld x$ has similar properties. A residuated poset is \emph{balanced} if it satisfies the identity $x\ld x \approx x\rd x$. The following proposition gives a number of conditions equivalent to being balanced.

\begin{proposition}\label{prop:equivalent:char:balanced}
The conditions below are equivalent in any residuated poset~$\m A$.
\begin{enumerate}[(1)]
    \item $\mathbf{A}\models x\ld x \approx x\rd x$.
    \item $\lscr pA = A_p$, for every $p\in \Idp\m A$.
    \item $p\ld A =  A\rd p$, for every $p\in \Idp\m A$.
    \item $p\ld a = a$ if and only if $a\rd p = a$, for every $a\in A$ and every $p\in\Idp\m A$.
    \item $pa = a$ if and only if $ap = a$, for every $a\in A$ and every $p\in \Idp\m A$.
    \item $p\ld a = a\rd p$, for every $a\in A$ and every $p\in\Idp\m A$.
    \item $pa = ap$, for every $a\in A$ and every $p\in \Idp{\m A}$.
    \item $\m A\models (x\ld x)y\approx y(x\ld x)$. 
    \item $\m A\models (x\rd x)y\approx y(x\rd x)$.    
\end{enumerate}
If these conditions are satisfied, then $\pair{\Idp\m A,\cdot}$ is a join semilattice with least element $1$ and its induced order is the restriction of the order of $\m A$.
\end{proposition}

\begin{proof}
\begin{itemize}[align=left, left=0pt .. \parindent]
\item[(1) $\Leftrightarrow$ (2)] This follows from the definitions of $\lscr pA$ and $A_p$ and Lemma~\ref{lem:IdpA:identities}.

\item[(2) $\Leftrightarrow$ (3)] This is a consequence of Lemma~\ref{char:interdefinability}.

\item[(3) $\Leftrightarrow$ (4) $\Leftrightarrow$ (5)] This follows from Lemma~\ref{lem:char:A/p=Ap}.

\item[(4) $\Leftrightarrow$ (6)] From~(4) and the fact that $p\ld (p\ld a) = p\ld a$ and $(a\rd p)\rd p = a\rd p$, we deduce that $p\ld a\rd p = p\ld a$ and $p\ld a\rd p = a\rd p$, respectively. Therefore, $p\ld a = a\rd p$. The reverse implication is immediate.

\item[(5) $\Leftrightarrow$ (7)] From~(5) and the fact that $ppa = pa$ and $app = ap$, we deduce that $pap = pa$ and $pap = ap$, respectively. Therefore, $pa = ap$. The reverse implication is immediate.

\item[(7) $\Leftrightarrow$ (8)] This follows from Lemma~\ref{lem:IdpA:identities}.

\item[(7) $\Leftrightarrow$ (9)] This also follows from Lemma~\ref{lem:IdpA:identities}.
\end{itemize}

\medskip

By part~7, we have that $\pair{\Idp\m A,\cdot}$ is a join semilattice. Moreover, notice that for every $p,q\in\Idp\m A$ we have that $p\le q$ if and only if $p\cdot q = q$, and thus the order induced by $\pair{\Idp\m A,\cdot}$ is the restriction of the order of $\m A$. Hence, the least element of $\pair{\Idp\m A,\cdot}$ is~$1$.\qed
\end{proof}

\begin{lemma}\label{lem:inequalities:balanced}
The following inequations hold in all balanced residuated posets.
\[
1_x \le 1_{xy},\quad
1_x \le 1_{yx},\quad
1_x \le 1_{x\rd y},\quad
1_x \le 1_{y\ld x}.
\]
\end{lemma}

\begin{proof}
For all $a,b\in A$, the equality $1_a\cdot ab = ab$ implies that $1_a\le ab\rd ab = 1_{ab}$. And, analogously, $1_a\cdot ba = b\cdot 1_a\cdot a = ba$ implies that $1_a\le ba\rd ba = 1_{ba}$. Also, the inequality $(a\rd b)b \le a$ implies that
\[
1_a = a\rd a \le a\rd ((a\rd b)b) = (a\rd b)\rd (a\rd b) = 1_{a\rd b}.
\]
The proof of the last inequation is similar.\qed
\end{proof}

As we have seen in Proposition~\ref{prop:equivalent:char:balanced}, the two partitions $\{\lscr pA : p\in\Idp\m A\}$ and $\{A_p : p\in\Idp\m A\}$ coincide if $\m A$ is balanced. We shall now investigate the case where all these equivalence classes are universes of subalgebras of the algebra~$\m A$. Obviously, the constant $1$ will lie in only one of these components, so the actual question is when the equivalence classes are the universes of algebras of the same type, such that their constant-free reducts are subalgebras of the constant-free reduct of $\m A$. We need to impose conditions which ensure that each one of these classes is closed under products and residuals:
\begin{enumerate}[(H1),leftmargin=*]
    \item\label{H1} $1_x \approx 1_y \implies 1_{xy} \approx 1_x$,
    \item\label{H2} $1_x \approx 1_y \implies 1_{x\rd y} \approx 1_x$,
    \item\label{H3} $1_x \approx 1_y \implies 1_{x\ld y} \approx 1_x$.
\end{enumerate}

These conditions are not completely independent, as we now show.

\begin{proposition}\label{prop:H2:iff:H3:imply:H1}
In the po-variety of balanced residuated posets, the quasiequations~\ref{H2} and~\ref{H3} are equivalent, and both imply~\ref{H1}.
\end{proposition}

\begin{proof}
In order to prove that~\ref{H2} implies~\ref{H3}, suppose that $a,b\in A$ are such that $1_a = 1_b$. Then,
\[
1_{a\ld b} \le 1_{a\ld b\rd b} = 1_{a\ld a\rd a} = 1_{1_a\rd a} = 1_a = 1_b \le 1_{a\ld b},
\]
where the inequalities follow from Lemma~\ref{lem:inequalities:balanced}, and the third equality follows from~\ref{H2}, since $1_{1_a} = 1_a$. The proof that~\ref{H3} implies~\ref{H2} is analogous.

In order to prove that~\ref{H2} implies~\ref{H1}, suppose again that $a,b\in A$ are such that $1_a = 1_b$. Since $a(a\ld 1_a)a \le 1_a\cdot a = a$, we deduce that $(a\ld 1_a)a \le a\ld a = 1_a = 1_b = b\rd b$, and therefore $(a\ld 1_a)ab \le b$. Hence, $ab \le (a\ld 1_a)\ld b$. Thus,
\begin{align*}
1_{ab} &= ab\rd ab \le (a\ld 1_a)\ld b\rd ab = (a\ld 1_a)\ld (b\rd b) \rd a = (a\ld 1_a)\ld 1_b \rd a \\ 
&= (a\ld a\rd a)\ld (1_a \rd a) = (1_a\rd a)\ld (1_a\rd a) = 1_{1_a\rd a} = 1_a \le 1_{ab},
\end{align*}
where the last equalities follow from~\ref{H2} and the fact that $1_{1_a}=1_a$ and the last inequality follows from Lemma~\ref{lem:inequalities:balanced}.\qed
\end{proof}

For every balanced residuated poset $\m A$ and every $p\in\Idp\m A$, we can consider the partial structure $\m A_p = \pair{A_p,\le_p,\cdot_p,\ld_p,\rd_p,1_p}$, where the relation $\le_p$ and each of its operations are the restrictions of the corresponding operations of $\m A$. We call $\m A_p$ the \emph{$p$-component} of $\m A$, for every $p\in\Idp\m A$. Conditions~\ref{H1}--\ref{H3} ensure that all the operations are total in all the components.

\begin{proposition}\label{prop:semilattice:sum:decomp:H123}
Every balanced residuated poset satisfying~\ref{H1}--\ref{H3} decomposes as a family of disjoint residuated posets $\{\m A_p : p\in\Idp\m A\}$ such that the constant-free reduct of $\m A_p$ is a subalgebra of the constant-free reduct of $\m A$ and $1_p$ is the only positive idempotent of $\m A_p$. 
\end{proposition}

\begin{proof}
The partition $\{A_p : p\in\Idp\m A\}$ is determined by the equivalence relation given by $a\equiv b$ if and only if $1_a = 1_b$. Therefore, the quasiequations~\ref{H1}--\ref{H3} express the conditions that every equivalence class $A_p$ is closed under products and residuals. Moreover, for every $a\in A_p$ we have that $1_p\cdot a = pa = a$, and $\m A_p$ inherits from $\m A$ all the properties to be a residuated poset.\qed
\end{proof}

The residuated posets $\m A_p$ are \emph{integrally closed} \cite{IntegrallyclosedRL} by construction: they satisfy the equation $x \ld x \approx 1$, or equivalently the equation $x \rd x \approx 1$. Both of these equations are equivalent to $1$ being the only positive idempotent, since $a \ld a$ and $a \rd a$ are positive idempotents for each element~$a$, and conversely $a \ld a = a = a \rd a$ for each positive idempotent~$a$. The equivalence between the equations $x \ld x \approx 1$ and $x \rd x \approx 1$ also follows directly from the equations $(x \rd x) \ld (x \rd x) \approx x \rd x$ and $x \ld x \approx (x \ld x) \rd (x \ld x)$ which hold in all residuated posets. Notice that an integrally closed residuated poset cannot be non-trivially decomposed in the above manner.

\begin{example}\label{ex:RP:H123}
The left poset $\pair{A,\le}$ of Fig.~\ref{fig:H123:but:not:H456} can be equipped with a commutative idempotent multiplication, namely, the meet operation of the poset $\langle A,\sqsubseteq\rangle$ on the right. One can check that this multiplication preserves all joins of $\pair{A,\le}$ and therefore it is residuated. In this way, we obtain a commutative, and therefore balanced, residuated poset $\m A$, where $\Idp\m A = \{1,p,q\}$ and $A_1 = \{1,a\}$, $A_p = \{p\}$, and $A_q = \{q,b,\bot\}$. These sets are indeed closed under residuals, and therefore $\m A$ satisfies~\ref{H1}--\ref{H3}, by Proposition~\ref{prop:H2:iff:H3:imply:H1}.
\end{example}

\begin{figure}[ht]\centering
\begin{tikzpicture}[baseline=0pt]
\node at (0,-2)[n]{$\le$};
\node(4) at (0,3)[label=left:$q$]{};
\node(3) at (0,2)[label=left:$p$]{} edge (4);
\node(2) at (1,1)[label=right:$b$]{} edge (3);
\node(1) at (-1,1)[label=left:$1$]{} edge (3);
\node(0) at (0,0)[label=right:$a$]{} edge (1) edge (2);
\node(-1) at (0,-1)[label=right:$\bot$]{} edge (0);
\end{tikzpicture}
\qquad\qquad
\begin{tikzpicture}[baseline=0pt]
\node at (0,-2)[n]{$\sle$};
\node(4) at (0,3)[label=right:$1$]{};
\node(3) at (1,2)[label=right:$p$]{} edge (4);
\node(2) at (-1,1.5)[label=left:$a$]{} edge (4);
\node(1) at (1,1)[label=right:$q$]{} edge (3);
\node(0) at (0,0)[label=right:$b$]{} edge (1) edge (2);
\node(-1) at (0,-1)[label=right:$\bot$]{} edge (0);
\end{tikzpicture}

\caption{Residuated poset satisfying~\ref{H1}--\ref{H3}.}
\label{fig:H123:but:not:H456}
\end{figure}
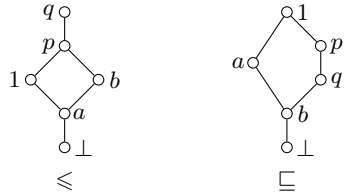

\section{P\l{}onka Sums of Directed Systems of Metamorphisms}

In the last section we obtained a ``decomposition result'', establishing that under certain minimal conditions, balanced residuated posets decompose as families of subalgebras without nontrivial positive idempotents. If we want to obtain a ``composition result'', everything points towards the use of P\l{}onka-style constructions. However, it turns out that the use of a single partition function is not sufficient to treat all the operations at once.

Let us start with the monoidal part. A natural candidate for a partition function $\odot\: A^2\to A$ that could take care of the monoidal reduct of a balanced residuated poset~$\m A$ is
\[
a\odot b = 1_b\cdot a.
\]
For this to be a partition function for the monoidal reduct of $\m A$, it should in particular satisfy~\ref{PF5} for the multiplication and for the particular value $b = 1$. That is, $1_{a_1\cdot a_2}\cdot 1 = 1_{a_2}\cdot 1_{a_1}\cdot 1$. We should then constrain ourselves to the po\nbd-subvariety of balanced residuated posets satisfying the equation
\begin{enumerate}[resume*]
    \item\label{H4} $1_x\cdot 1_y \approx 1_{xy}$.
\end{enumerate}
Notice that every balanced residuated poset $\m A$ satisfying~\ref{H4} also satisfies~\ref{H1}. Indeed, if $a,b\in A$ are such that $1_a = 1_b$, then $1_{ab} = 1_a\cdot 1_b = 1_a\cdot 1_a = 1_a$, by a direct application of~\ref{H4}. Therefore, in this case every equivalence class $A_p$ is closed under products, for every $p\in\Idp\m A$.

\begin{proposition}\label{prop:H4:monoidal:decomposition}
For every balanced residuated poset satisfying~\ref{H4}, the operation $\odot\: A^2\to A$ determined by $a\odot b = 1_b\cdot a$ is a partition function for the monoidal reduct of $\m A$. The corresponding partition of $A$ is $\{A_p : p\in\Idp\m A\}$ and the associated family of maps is $\Phi = \{\varphi_{pq}\: A_p \to A_q : p\le q\}$ given by $\varphi_{pq}(a) = qa$. All these are monoidal homomorphisms between the appropriate algebras and the monoidal reduct of $\m A$ is the P\l{}onka sum of the family $\Phi$.
\end{proposition}

\begin{proof}
\ref{PF1}~holds because $a\odot a = 1_a\cdot a = a$. As for~\ref{PF2}, notice that 
\[
a\odot (b\odot c) = 1_{1_c\cdot b}\cdot a = 1_{1_c}\cdot 1_b\cdot a = 1_c\cdot 1_b\cdot a = (a\odot b)\odot c.
\]
Concerning~\ref{PF3},
\[
a\odot b\odot c = 1_c\cdot 1_b\cdot a = 1_b\cdot 1_c\cdot a = a\odot c\odot b.
\]
As for~\ref{PF4}, we only need to check that the property holds for $\sigma = {\cdot}$. Indeed,
\[
(a_1\cdot a_2)\odot b = 1_b\cdot (a_1\cdot a_2) = 1_b\cdot 1_b\cdot a_1\cdot a_2 = 1_b\cdot a_1 \cdot 1_b\cdot a_2 = (a_1\odot b)\cdot (a_2\odot b).
\]
Concerning~\ref{PF5}, $b\odot (a_1\cdot a_2) = 1_{a_1\cdot a_2}\cdot b = 1_{a_1}\cdot 1_{a_2}\cdot b = 1_{a_2}\cdot 1_{a_1}\cdot b = b\odot a_1\odot a_2$. And for~\ref{PF6}, we just need to check that $a\odot 1 = 1_1\cdot a = 1\cdot a = a$.

Two elements $a,b\in A$ are equivalent, according to the partition function $\odot$ if and only if $a\odot b = a$ and $b\odot a = b$, that is, $1_b\cdot a = a$ and $1_a\cdot b = b$, and in particular $1_b\le 1_a$ and $1_a\le 1_b$, that is, $1_a = 1_b$. On the other hand, if $1_a=1_b$, then $a\odot b = 1_b\cdot a = 1_a\cdot a = a$ and analogously, $b\odot a = b$. Notice also that if there is some $a\in A_p$ and $b\in A_q$ so that $b\odot a = b$, then $1_a\cdot b =b$, whence $p = 1_a \le b\rd b = 1_b = q$. Hence, the partition corresponding to $\odot$ is $\{A_p : p\in\Idp\m A\}$ and the order of the set of indices $\Idp\m A$ determined by the partition function $\odot$ coincides with the restriction of the order of $\m A$.

Finally, given positive idempotents $p\le q$, the map $\varphi_{pq}\: A_p\to A_q$ determined by the partition function $\odot$ is given by $\varphi_{pq}(a) = a\odot q = 1_q\cdot a = qa$. The rest follows from Theorem~\ref{thm:Plonka:Theorem}.\qed
\end{proof}

The residuals of balanced residuated posets resist such a direct approach. Indeed, the most natural choice for a partition function $\otimes\: A^2\to A$ that could take care of the residuals would be
\[
a\otimes b = 1_b\ld a = a\rd 1_b.
\]
For this to be a partition function with respect to the residuals, $\m A$ should satisfy the equation $1_x\ld (1\ld 1) = (1_x\ld 1)\ld (1_x\ld 1)$, as a particular case of~\ref{PF4}. However, the next lemma shows that the only balanced residuated posets that satisfy this equation are those with no nontrivial positive idempotents. Therefore, this condition is overly restrictive and will need to be relaxed.

\begin{lemma}
A balanced residuated poset satisfies $1_x\ld 1 \approx (1_x\ld 1)\ld (1_x\ld 1)$ if and only if it satisfies $1_x \approx 1$.   
\end{lemma}

\begin{proof}
Let $p$ be a positive idempotent satisfying that $p\ld 1 = (p\ld 1)\ld (p\ld 1) = 1_{p\ld 1}$. Since $1\le 1_{p\ld 1} = p\ld 1$, therefore $p\le 1$, whence we deduce that $p = 1$. Hence, $1_x\ld 1 \approx (1_x\ld 1)\ld (1_x\ld 1)$ implies $1_x \approx 1$. The reverse implication is immediate.\qed
\end{proof}

Nonetheless, under less restrictive conditions, the above operation $\otimes$ satisfies the remaining properties of a partition function. Let us find out these conditions. We would like~\ref{PF5} to be satisfied for arbitrary elements $a_1$ and $a_2$ and the particular value $b = 1$, namely, $1_{a_1\ld a_2}\ld 1 = 1_{a_2}\ld (1_{a_1} \ld 1) = (1_{a_1}\cdot 1_{a_2})\ld 1$. We also need the corresponding condition for the other residual, and both are true if the following equations are satisfied:
\begin{enumerate}[resume*]
    \item\label{H5} $1_{x\ld y} \approx 1_x\cdot 1_y$,
    \item\label{H6} $1_{x\rd y} \approx 1_x\cdot 1_y$.
\end{enumerate}
Notice that both equations imply~\ref{H2} and~\ref{H3} and, as a consequence, every $A_p$ is closed under residuals if $\m A$ satisfies~\ref{H5} and~\ref{H6}. On the other hand, the three properties~\ref{H1}--\ref{H3} do not imply~\ref{H4}--\ref{H6}. Indeed, the residuated poset of Example~\ref{ex:RP:H123} fails~\ref{H4}--\ref{H6}, since $pa = b$ and $1_{pa} = 1_b = q \neq p\cdot 1 = 1_p\cdot 1_a$, and $p\ld a = \bot$ and $1_{p\ld a} = 1_\bot = q \neq p\cdot 1 =1_p\cdot 1_a$.

\begin{proposition}\label{prop:H56:left:normal:band:for:residuals}
For every balanced residuated poset satisfying~\ref{H5} and~\ref{H6}, the function $\otimes\: A^2\to A$ determined by $a\otimes b = 1_b\ld a$ is a left normal band, which satisfies~\ref{PF5} for the residuals of $\m A$.
\end{proposition}

\begin{proof}
\ref{PF1} holds because $a\otimes a = 1_a\ld a = a$. As for~\ref{PF2}, notice that
\[
a\otimes(b\otimes c) = 1_{1_c\ld b}\ld a = (1_{1_c}\cdot 1_b)\ld a = (1_b\cdot 1_c)\ld a = 1_c\ld (1_b\ld a) = (a\otimes b)\otimes c,
\]
where the second equality follows from~\ref{H5}. It will also follow from~\ref{H6}, since $1_c\ld b = b\rd 1_c$.
Concerning~\ref{PF3},
\[
a\otimes b\otimes c = 1_c\ld (1_b \ld a) = (1_b\cdot 1_c)\ld a = (1_c\cdot 1_b)\ld a = 1_b\ld (1_c\ld a) = a\otimes c\otimes b.
\]
As for~\ref{PF5}, we need to check two equalities. The first equality is
\[
b\otimes (a_1\ld a_2) = 1_{a_1\ld a_2}\ld b = (1_{a_1}\cdot 1_{a_2})\ld b = 1_{a_2}\ld (1_{a_1}\ld b) = b\otimes a_1\otimes a_2,
\]
which is a consequence of~\ref{H5}. And the second equality is $b\otimes (a_1\rd a_2) = b\otimes a_1\otimes a_2$, which is a consequence of~\ref{H6}.\qed % And finally, for~\ref{PF6} we only need to check that $a\otimes 1 = 1_1\ld a = 1\ld a = a$.\qed
\end{proof}

Since we are looking for a generalization of Theorem~\ref{thm:Plonka:Theorem}, let us analyze its different parts separately. We will divide the proof of P\l{}onka's decomposition theorem into a number of lemmas in order to clarify which assumptions are used in the proof of every claim. The proofs themselves can be found, in one way or another, in the literature (see~\cite{BPP22} and~\cite{GR91}, for instance). Given an operation $\odot\: A^2\to A$, we can define the binary relations $\le_\odot$ and $\equiv_\odot$ as follows:
\[
a\le_\odot b \iff\ b\odot a = b \qquad\text{and}\qquad a\equiv_\odot b \iff a\le_\odot b \text{ and }\space b\le_\odot a.
\]

\begin{lemma}\label{lem:PF123:join:semilattice}
For every left normal band $\odot\: A^2\to A$, the relation $\le_\odot$ is a preorder on $A$ compatible with $\odot$. Therefore, $\equiv_\odot$ is a congruence of $\pair{A,\odot}$ and the induced relation $\le'_\odot$ is a partial order on $A/{\equiv}_\odot$. Moreover, $\m S^\odot = \pair{A,\odot}/{\equiv}_\odot$ is a join semilattice.
\end{lemma}

If a balanced residuated poset satisfies~\ref{H4}--\ref{H6}, then we have two left normal bands on the same set $A$: $\odot$ and $\otimes$, by Propositions~\ref{prop:H4:monoidal:decomposition} and~\ref{prop:H56:left:normal:band:for:residuals}. Notice that in general they are different.\footnote{Indeed, one can readily see that $\odot$ and $\otimes$ are equal if and only if $\m A$ satisfies $1_x\approx 1$, and this forces $a\odot b = a\otimes b = a$, for all $a,b\in A$.} However, they induce the same equivalence relation. We already saw that $a\equiv_\odot b$ if and only if $1_a = 1_b$. As for $\otimes$, notice that 
\[
a\le_\otimes b \iff b\otimes a = b \iff 1_a\ld b = b \iff 1_a\le b\rd b = 1_b.
\]
Hence, $a\equiv_\otimes b \iff 1_a = 1_b$. In what follows, we will say that two left normal bands $\odot$ and $\otimes$ on the same set $A$ are \emph{compatible} if ${\equiv}_\odot = {\equiv}_\otimes$.

Under the conditions of the previous lemma, let us choose a set $I\subset A$ of representatives of the equivalence classes of $\equiv_\odot$ and let $\m I =\pair{I, \lor}$ be the transport of the structure of $\m S^\odot$ to $I$. That is, for every $p,q\in I$, $p\lor q$ is the representative of the equivalence class $[p\odot q]$. Thus, $\m I = \pair{I,\lor} \cong \m S^\odot$ and the induced order in $I$ is the restriction of $\le_\odot$, which we denote in the same way. Let $A_p$ be the equivalence class $[p]$, for every $p\in I$.

\begin{lemma}\label{lem:PF123:phis:directed:system}
Let $\m A$ be an algebra with a left normal band $\odot$.
\begin{enumerate}[(1)]
\item For every $p\le_\odot q$ in $\m I$, the map $\varphi_{pq}\: A_p\to A_q$ given by $\varphi_{pq}(a) = a\odot q$ is well defined.

\item The family $\Phi = \{\varphi_{pq} : p\le_\odot q\}$ is a semilattice directed system of maps.
\end{enumerate}
\end{lemma}

\begin{lemma}\label{lem:A:PF5:closure:under:sigma}
Let $\m A$ be an algebra with a left normal band $\odot$.
\begin{enumerate}[(1)]
\item If $\sigma^\m A$ is an $n$-ary operation which satisfies~\ref{PF5}, then the following hold.

\begin{enumerate}
\item For all $a_1,\dots, a_n\in A$, we have that $\sigma^\m A(a_1,\dots,a_n) \equiv_\odot a_1\odot\dots\odot a_n$.

\item Every equivalence class $A_p$ is closed under $\sigma^\m A$.
\end{enumerate}    

\item If $\omega^\m A$ is a constant which satisfies~\ref{PF6}, then the class $[\omega^\m A]$ is the least element of $\m S^\odot$.
\end{enumerate}
\end{lemma}

\begin{proof}
For the first claim, notice that
\begin{align*}
\sigma^\m A(a_1,\dots,a_n) \odot a_i &= \sigma^\m A(a_1,\dots,a_n) \odot \sigma^\m A(a_1,\dots,a_n) \odot a_i \\
    &= \sigma^\m A(a_1,\dots,a_n) \odot a_1 \odot \dots \odot a_n\odot a_i\\
    &= \sigma^\m A(a_1,\dots,a_n) \odot a_1 \odot \dots \odot a_n\\
    &= \sigma^\m A(a_1,\dots,a_n) \odot \sigma^\m A(a_1,\dots,a_n)\\
    &= \sigma^\m A(a_1,\dots,a_n).
\end{align*}
The first and the last equalities follow from~\ref{PF1} and the second and fourth from~\ref{PF5}. The third equality is a consequence of a number of applications of~\ref{PF2} and~\ref{PF3} and one application of~\ref{PF1}. Therefore, $a_i\le_\odot  \sigma^\m A(a_1,\dots,a_n)$ and thus $a_1\odot\dots\odot a_n \le_\odot \sigma^\m A(a_1,\dots,a_n)$. For the reverse inequality, take $q\in I$ so that $q \equiv_\odot a_1\odot\dots\odot a_n$. Then,
\[
q\odot\sigma^\m A(a_1,\dots,a_n) = q\odot a_1\odot \dots\odot a_n \equiv_\odot q\odot q = q,
\]
and hence $\sigma^\m A(a_1,\dots,a_n) \le_\odot q \equiv_\odot a_1\odot\dots\odot a_n$. The second part follows immediately, because if $a_1,\dots,a_n\in A_p$ then 
\[
\sigma^\m A(a_1,\dots,a_n) = a_1\odot \dots\odot a_n \equiv_\odot p\odot \dots\odot p = p,
\]
that is, $\sigma^\m A(a_1,\dots,a_n) \in A_p$. Finally, if $a\odot \omega^\m A = a$ holds for all $a\in A$, then $\omega^\m A \le_\odot a$, whence the last claim follows.\qed   
\end{proof}

Thus, we see that the only part of Theorem~\ref{thm:Plonka:Theorem} that requires~\ref{PF4} is the fact that the semilattice directed family of maps is actually a family of homomorphisms. This is precisely the part that we will relax in our generalization of P\l{}onka sum.

A \emph{partition system} for an algebra $\m A$ is an assignment $\sigma\mapsto\pair{\odot^{\sigma}_0,\dots,\odot^{\sigma}_n}$ for every $n$-ary operation or constant symbol $\sigma$ of the type of $\m A$ to tuples of elements of a set $O$ of compatible left normal bands on $A$, such that for every $n$-ary operation symbol $\sigma$, every constant symbol $\omega$, and all $a_1,\dots,a_n,b\in A$,
\begin{enumerate}[align=left, left=0pt .. \parindent]
\myitem[(PF4$^\sigma$)]\label{PF4s} $\sigma^\m A(a_1,\dots,a_n)\odot^\sigma_0 b = \sigma^\m A(a_1\odot^{\sigma}_1 b,\dots,a_n\odot^{\sigma}_n b)$,

\myitem[(PF5$^\sigma$)]\label{PF5s} $b\odot^\sigma_0 \sigma^{\mathbf{A}}(a_1,\dots,a_n) = b\odot^\sigma_0 a_{1}\odot^\sigma_0 \dots\odot^\sigma_0 a_n $,

\myitem[(PF6$^\omega$)]\label{PF6w} $b\odot^\omega_0 \omega^\m A = b$.
\end{enumerate}

Given two algebras $\m A$ and $\m B$ of the same type $\tau$, a \emph{metamorphism} $f\:\m A\meta \m B$ is a map $f\:\tau\to \big(B^A\big)^*$, $\sigma\mapsto f^\sigma = \pair{f^{\sigma 0},\dots,f^{\sigma n}}$, where $n$ is the arity of $\sigma$, such that
\[
f^{\sigma 0}(\sigma^\m A(a_1,\dots,a_n)) = \sigma^\m B(f^{\sigma 1}(a_1),\dots, f^{\sigma n}(a_n)).
\]
In particular, for every constant $\omega$, $f^\omega = \pair{f^{\omega 0}}$ and $f^{\omega 0}(\omega^\m A) = \omega^\m B$. The \emph{composition} of two metamorphisms $f\:\m A\meta\m B$ and $g\:\m B\meta\m C$ is the metamorphism $g\circ f\:\m A\meta\m C$ defined by $(g\circ f)^\sigma = \pair{g^{\sigma 0}\circ f^{\sigma 0},\dots,g^{\sigma n}\circ f^{\sigma n}}$. The \emph{identity} metamorphism on $\m A$ is $id^\m A\:\m A\meta \m A$ defined by $id^\sigma = \pair{id,\dots,id}$. The algebras of type $\tau$ and metamorphisms form a category $\Meta^\tau$. Every homomorphism $f\:\m A\to\m B$ determines a metamorphism $\dot f\:\m A\meta\m B$ such that for every $n$-ary operation or constant symbol $\sigma$, $\dot f^{\sigma 0} = \dots = \dot f^{\sigma n} = f$. This assignment defines a faithful functor from the category $\Alg^\tau$ of algebras of type~$\tau$ and homomorphisms to $\Meta^\tau$ which is the identity on objects.

A \emph{semilattice directed system of metamorphisms} is a functor $\Xi\: \m I\to\Meta^\tau$ where $\m I = \pair{I,\lor}$ is a join semilattice. It can also be seen as a family $\Xi = \{{\xi_{pq}\:\m A_p\meta\m A_q} : {p\le q \text{ in } \m I}\}$ such that $\xi_{pp} = id^{\m A_p}$ and $\xi_{qr}\circ\xi_{pq} = \xi_{pr}$. Note that every directed system of homomorphisms can be viewed as a directed system of metamorphisms, via the composition with the inclusion $\Alg^\tau\to\Meta^\tau$.

The \emph{P\l{}onka sum} of a directed system of metamorphisms~$\Xi$ is the algebra $\m A$ whose universe is $A = \biguplus A_p$, and such that for every $n$-ary operation $\sigma$ and all $a_1\in A_{p_1}, \dots, a_n\in A_{p_n}$,
\[
\sigma^\m A(a_1,\dots,a_n) = \sigma^{\m A_q}(\xi^{\sigma 1}_{p_1q}(a_1),\dots,\xi^{\sigma n}_{p_nq}(a_n)),
\]
where $q = p_1\lor\dots\lor p_n$. Consequently, if the type $\tau$ contains a constant symbol~$\omega$, then we assume that $\m I$ has a least element $\bot$ and $\omega^\m A = \omega^{\m A_\bot}$. It follows immediately from its definition that the constant-free reduct of $\m A_p$ is a subalgebra of the constant-free reduct of $\m A$, and if $\m A$ has constants then $\m A_\bot$ is a subalgebra of~$\m A$. We have all the machinery for our main results about P\l{}onka sums of directed systems of metamorphisms, of which Theorem~\ref{thm:Plonka:Theorem} is an immediate consequence. The proofs are straightforward and we omit them for lack of space.

\begin{restatable}{theorem}{DecompositionThm}
%\begin{theorem}
\label{thm:Plonka:decomposition:meta}
Let $\m A$ be an algebra with a partition system, and $\equiv$ and $\m I = \pair{I,\lor}$ their induced equivalence relation and corresponding join semilattice. 
\begin{enumerate}[(1), leftmargin=*]
\item Every equivalence class $A_p$ of $\equiv$ is the universe of an algebra $\m A_p$ of the same type whose constant-free reduct is a subalgebra of the constant-free reduct of~$\m A$, and $\omega^{\m A_p} = \omega^{\m A}\odot^\omega_0 p$ for every constant symbol $\omega$.

\item For all $p\le q$ in $\m I$, there is a metamorphism $\xi_{pq}\:\m A_p\meta\m A_q$ defined by
\[
\xi^{\sigma i}_{pq}(a) = a\odot^\sigma_i q, \quad \text{for every $n$-ary symbol $\sigma$ and $i\le n$}.
\]

\item The family $\Xi = \{\xi_{pq}\:\m A_p\meta\m A_q : p\le q \text{ in }\m I\}$ is a semilattice directed system of metamorphisms and $\m A$ is its P\l{}onka sum.
\end{enumerate}
%\end{theorem}
\end{restatable}\unskip

\begin{restatable}{theorem}{DirectedSystemFromPartition}
%\begin{theorem}
Every semilattice directed system of metamorphisms is induced by a partition system for its P\l{}onka sum $\m A$.
%\end{theorem}
\end{restatable}

Within this framework, we can understand the decompositions of the algebraic reducts of residuated posets satisfying~\ref{H4}--\ref{H6}. By \emph{algebraic reduct} we mean the algebra resulting from removing the partial order.

\begin{theorem}\label{thm:H456:partition:meta}
Given a balanced residuated poset $\m A$ satisfying~\ref{H4}--\ref{H6}, the set $O = \{\odot,\otimes\}$ of left normal bands defined by $a\odot b = 1_b\cdot a$ and $a\otimes b = 1_b\ld a$, and the assignment $1\mapsto \odot$, $\cdot\mapsto\pair{\odot,\odot,\odot}$, $\ld\mapsto\pair{\otimes,\odot,\otimes}$, and $\rd\mapsto\pair{\otimes,\otimes,\odot}$ define a partition system for the algebraic reduct of $\m A$.
\end{theorem}

\begin{proof}
First, $\odot$ and $\otimes$ are compatible left normal bands, $\odot$ satisfies~\hyperref[PF5s]{(PF5$^\cdot$)} and~\hyperref[PF6w]{(PF6$^1$)}, and $\otimes$ satisfies~\hyperref[PF5s]{(PF5$^\ld$)} and~\hyperref[PF5s]{(PF5$^\rd$)}, by Propositions~\ref{prop:H4:monoidal:decomposition} and~\ref{prop:H56:left:normal:band:for:residuals}. Moreover, $\odot$ also satisfies~\hyperref[PF4s]{(PF4$^\cdot$)}. Also, for all $a_1,a_2,b\in A$,
\begin{align*}
(a_1\ld a_2)\otimes b &= 1_b\ld(a_1\ld a_2) = (a_1\cdot 1_b)\ld a_2 = (1_b1_b\cdot a_1)\ld a_2 \\
&= (1_b\cdot a_1)\ld (1_b\ld a_2) = (a_1\odot b) \ld (a_2\otimes b).    
\end{align*}
And, analogously, $(a_1\rd a_2)\otimes b = (a_1\otimes b)\rd (a_2\odot b)$.\qed
\end{proof}

\begin{corollary}\label{cor:residuated:posets:H456:are:Plonka:sums:meta}
The algebraic reduct of every residuated poset $\m A$ satisfying~\ref{H4}--\ref{H6} is the P\l{}onka sum of the directed system of metamorphisms between its components $\Xi = \{\xi_{pq}\: \m A_p\meta\m A_q : p\le q \text{ in }\m I\}$ given by 
\begin{align*}
\xi^1_{pq} &= \pair{\varphi_{pq}}, &
\xi^\cdot_{pq} &= \pair{\varphi_{pq},\varphi_{pq},\varphi_{pq}}, \\
\xi^\ld_{pq} &= \pair{\psi_{pq}, \varphi_{pq}, \psi_{pq}}, &
\xi^\rd_{pq} &= \pair{\psi_{pq}, \psi_{pq}, \varphi_{pq}}.
\end{align*}
where $\varphi_{pq},\psi_{pq}\: A_p\to A_q$ are defined by $\varphi_{pq}(a)=qa$ and $\psi_{pq}(a) = q\ld a$. Moreover, for all $p,q\in I$, $a\in A_p$, and $b\in A_q$,
\begin{equation*}
a\le b \quad\iff\quad \varphi_{ps}(a) \le_s \psi_{qs}(b),\quad \text{where } s=pq.
\end{equation*} 
\end{corollary}

\begin{proof}
The first part is a consequence of Theorems~\ref{thm:Plonka:decomposition:meta} and~\ref{thm:H456:partition:meta}. The remaining part holds since if $a,b\in A$, $a\in A_p$, $b\in A_q$, and $s=pq$, then
\begin{align*}
a\le b &\iff 1\le a\ld b =\varphi_{ps}(a)\ld_s\psi_{qs}(b)
\iff 1_s = s \le \varphi_{ps}(a)\ld_s\psi_{qs}(b) \\
&\iff \varphi_{ps}(a) \le_s \psi_{qs}(b).    \tag*{\qed}
\end{align*}
\end{proof}

\section{Sums of Posets}\label{sec:sums:of:posets}

In our last result we showed that the order $\le$ of $\m A$ can be reconstructed from the orders $\le_p$ of its components together with $\Phi$ and $\Psi$. 
In this section, we investigate this aspect of the reconstruction of the order of~$\m A$ in an isolated manner. More concretely, given a family $\{\m A_p : p\in I\}$ of disjoint posets, we would like to define an order $\le$ on the disjoint union $A = \biguplus A_p$ of their universes \emph{extending} each one of the partial orders, meaning that ${\le}\cap A_p^2 = {\le_p}$, for every $p\in I$. Obviously, we could just take ${\le} = \bigcup_p{\le_p}$, but this would be insufficient for our needs, since the families of posets in which we are interested are actually connected via some maps, and the order $\le$ should be compatible, in some sense, with these maps.

Consider a join-semilattice $\m I = \pair{I,\lor}$, and a pair $\pair{\Phi,\Psi}$ of directed systems of monotone maps $\Phi = \{\varphi_{pq}\: \m A_p \to \m A_q: p\le q \text{ in }\m I\}$ and $\Psi = \{\psi_{pq}\: \m A_p \to \m A_q: p\le q \text{ in }\m I\}$, and define the relation $\le$ on $A = \biguplus A_p$ as follows: for all $p,q\in I$, $a\in A_p$, and $b\in A_q$,
\begin{equation}\label{eq:def:sum:order}
a\le b \quad\iff\quad \varphi_{ps}(a) \le_s \psi_{qs}(b),\quad \text{where } s=p\lor q.    
\end{equation}
Now, this relation doesn't have to be an order in general, but it will be if the following three conditions are satisfied. In that case, we call $\pair{A,\le}$ the \emph{sum} of the family of posets $\{\m A_p :p\in I\}$ \emph{over} $\pair{\Phi,\Psi}$.

\begin{enumerate}[(O1),leftmargin=*]
\item\label{O1} if $p < q$ then $\psi_{pq} <_q \varphi_{pq}$ pointwise,

\item\label{O2} if $p \le q,r$ and $t = q\lor r$, then $\varphi_{qt}\psi_{pq} \le_t \psi_{rt}\varphi_{pr}$ pointwise,%\footnote{Notice that, taking $q=r$, \ref{O2}~implies that if $p\le q$ then $\psi_{pq}\le_q\varphi_{pq}$, pointwise, but we actually require the strict inequality of~\ref{O1}.}

\item\label{O3} for all $a,b\in A_p$ and $p\le q$, if $\varphi_{pq} (a) \le_q \psi_{pq} (b)$, then $a \le_p b$.
\end{enumerate}

\begin{theorem}\label{thm:char:sums:of:posets:over:directed:systems}
Given a pair $\pair{\Phi,\Psi}$ of directed systems of monotone maps, the relation $\le$ defined by~\eqref{eq:def:sum:order} is a partial order extending the order of each poset if and only if $\pair{\Phi,\Psi}$ satisfies~\ref{O1}--\ref{O3}. 
\end{theorem}

\begin{proof} 
Suppose that $\pair{\Phi,\Psi}$ satisfies~\ref{O1}--\ref{O3}. The fact that the restriction of $\le$ to $A_p$ is $\le_p$, for every $p\in I$, follows immediately from definition~\eqref{eq:def:sum:order} and the fact that that $\varphi_{pp} = \psi_{pp} = id_{A_p}$. Indeed, if $a,b\in A_p$ are such that $a\le b$, then $p = p\lor p$ and $a = \varphi_{pp}(a) \le_p \psi_{pp}(b) = b$. And reciprocally, if $a\le_p b$, then $\varphi_{pp}(a) = a \le_p b = \psi_{pp}(b)$, and therefore $a\le b$. In particular, reflexivity follows.

In order to prove antisymmetry, suppose that $a\in A_p$ and $b\in A_q$ are such that $a\le b$ and $b\le a$. Consider $s = p\lor q$ and suppose that $p\neq q$. Then, by the definition of $\le$ and condition~\ref{O1}, we have that
\[
\varphi_{ps}(a) \le_s \psi_{qs} (b) <_s \varphi_{qs} (b) \le_s \psi_{ps} (a) <_s \varphi_{ps} (a),
\]
which is impossible. Thus, $p = q = s$, $a\le_p b$, and $b\le_p a$, and hence $a=b$.

As for transitivity, suppose that $a\in A_p$, $b\in A_q$, and $c\in A_r$ are such that $a\le b\le c$. Consider $s = p\lor q$, $t = q\lor r$, $u = p\lor r$, $v = s\lor t$, and notice that $u \le v$. By definition of the order, we have that (*)~$\varphi_{ps} (a) \le_s \psi_{qs} (b)$ and (**)~$\varphi_{qt} (b) \le_t \psi_{rt} (c)$. Now,
\begin{align*}
\varphi_{uv} (\varphi_{pu} (a)) &= \varphi_{pv} (a)
	= \varphi_{sv} (\varphi_{ps} (a))      &&\text{by the compatibility of $\Phi$}\\
	&\le_v \varphi_{sv} (\psi_{qs} (b))	   &&\text{by the monotonicity of $\varphi_{sv}$ and (*)}\\
	&\le_v \psi_{tv} (\varphi_{qt} (b))    &&\text{by~\ref{O2}}\\
	&\le_v \psi_{tv} (\psi_{qt} (c))       &&\text{by the monotonicity of $\psi_{tv}$ and (**)}\\
	&= \psi_{qv} (c)
	= \psi_{uv} (\psi_{qu} (c)),           &&\text{by the compatibility of $\Psi$.}
\end{align*}
By~\ref{O3}, we deduce that $\varphi_{pu} (a) \le_u \psi_{qu} (c)$, that is $a \le c$.

In order to prove the reverse implication, suppose that the relation given by~\eqref{eq:def:sum:order} is a partial order extending the order of every $\m A_p$. First of all, notice that if $p\le q,r$ and $a\in A_p$, then $q = p\lor q$ and $\varphi_{qq}(\psi_{pq}(a) )\le_q \psi_{pq}(a)$, and therefore $\psi_{pq}(a)\le a$. Analogously, $a\le\varphi_{pr}(a)$, and therefore $\psi_{pq}(a)\le\varphi_{pr}(a)$, by transitivity of $\le$. That is, by definition~\eqref{eq:def:sum:order}, $\varphi_{qt}(\psi_{pq}(a)) \le_t \psi_{rt}(\varphi_{pr}(a))$, where $t = q\lor r$. That is, \ref{O2} holds. In particular, taking $q=r$, we have that $\psi_{pq}(a) \le_q \varphi_{pq}(a)$. But if $\varphi_{pq}(a) = \psi_{pq}(a)$ then we would have that $\psi_{pq}(a) = a$, by transitivity and antisymmetry of $\le$, which would only be possible if $p=q$. Hence, if $p<q$ then $\psi_{pq}(a) <_q \varphi_{pq}(a)$, which shows that~\ref{O1} holds too. Finally, if $a,b\in A_p$ and $p\le q$ are such that $\varphi_{pq}(a)\le_q\psi_{pq}(b)$, then $a \le \varphi_{pq}(a) \le \psi_{pq} (b) \le b$, as we have shown, and by transitivity we obtain that $a\le b$, and in particular $a\le_p b$.\qed 
\end{proof}

Our last result in this paper is a reciprocal to Corollary~\ref{cor:residuated:posets:H456:are:Plonka:sums:meta}. It tells us under which conditions the P\l{}onka sum of a directed system of metamorphisms can be endowed with a compatible order, making it a residuated poset.

\begin{theorem}\label{thm:construction:meta}
Let $\{\m A_p\: p\in I\}$ be a family of residuated posets indexed on a join semilattice $\m I = \pair{I,\lor}$ with least element $\bot$ and $\Phi,\Psi$ a pair of semilattice directed systems of monotone maps such that $\Xi=\{\xi_{pq}\:\m A_p\meta\m A_q : p\le q \text{ in }\m I\}$ is a directed system of metamorphisms defined by
\begin{align*}
\xi^1_{pq} &= \pair{\varphi_{pq}}, &
\xi^\cdot_{pq} &= \pair{\varphi_{pq},\varphi_{pq},\varphi_{pq}}, \\
\xi^\ld_{pq} &= \pair{\psi_{pq}, \varphi_{pq}, \psi_{pq}}, &
\xi^\rd_{pq} &= \pair{\psi_{pq}, \psi_{pq}, \varphi_{pq}}.
\end{align*}
and $\pair{\Phi,\Psi}$ satisfies~\ref{O1}--\ref{O3}. Then the P\l{}onka sum of $\Xi$ together with the sum of the poset reducts over $\pair{\Phi,\Psi}$ is a residuated poset.
\end{theorem}
\begin{proof}
$\pair{A,\cdot,1}$ is a monoid since by assumption $\Phi$ is a system of monoid homomorphisms. The relation $\le$ defined by \eqref{eq:def:sum:order} is a partial order by Theorem~\ref{thm:char:sums:of:posets:over:directed:systems}. For residuation (checking only $\ld$\,), let $a\in A_p$, $b\in A_q$, $c\in A_r$, $s=p\vee q$, $t=p\vee r$, and $u=s\vee r = q\vee t$. Then,
\[
\psi_{tu}(a\ld c) = \psi_{tu}\big(\varphi_{pt}(a)\ld_t\psi_{rt}(c)\big)
= \varphi_{tu}\varphi_{pt}(a)\ld_u\psi_{tu}\psi_{rt}(c)
= \varphi_{pu}(a)\ld_u\psi_{ru}(c).
\]
Hence,
\begin{align*}
    a\cdot b\le c&\iff \varphi_{ps}(a)\cdot_s\varphi_{qs}(b)\le c
    \iff \varphi_{pu}(a)\cdot_u\varphi_{qu}(b)\le_u \psi_{ru}(c)\\
    &\iff \varphi_{qu}(b)\le_u  \varphi_{pu}(a)\ld_u\psi_{ru}(c) = \psi_{tu}(a\ld c)
    \iff b\le a\ld c.\tag*{\qed}
\end{align*}
\end{proof}

\section{Instructive Examples}

\paragraph{\bfseries Involutive po-monoids.}
An \emph{involutive po-monoid}, or \emph{ipo-monoid}, (see~\cite{GFJiLo23}) is a structure of the form $\m A = \pair{A,\le,\cdot,1,\nein,\no}$ such that $\pair{A,\le}$ is a poset and $\pair{A,\cdot, 1}$ is a monoid satisfying
\[
x\le y\iff x\cdot\nein y\le \no 1 \iff \no y\cdot x\le \no 1. \tag{ineg}
\]
It is known that the product is residuated with residuals defined by $x\ld y = \nein(\no y\cdot x)$ and $x\rd y = \no(y\cdot\nein x)$. An ipo-monoid $\m A$ is \emph{locally integral} if it is balanced, and it satisfies $x\le 1_x$ and $x\ld 1_x = 1_x$, where $1_x = x\rd x = \no(x\cdot\nein x)$. In that case, it can be shown that $\odot,\otimes\: A^2\to A$ defined by $a\odot b = 1_b\cdot a$ and $a\otimes b = \no(1_b\cdot\nein a)$ are compatible left normal bands and the assignment $1\mapsto\odot$, $\cdot\mapsto\pair{\odot,\odot,\odot}$, $\nein\mapsto\pair{\otimes,\odot}$, and $\no\mapsto\pair{\otimes,\odot}$ is a partition system for~$\m A$. Hence, every locally integral ipo-monoid is a P\l{}onka sum, and its order can be recovered by~\eqref{eq:def:sum:order}. This is precisely the decomposition obtained in an ad hoc manner in~\cite{GFJiLo23}.

\paragraph{\bfseries P\l{}onka sum of two residuated posets.}
Consider two residuated posets $\m A_1$ and $\m A_2$ and two directed systems $\Phi=\{\varphi_{pq} : p\le q\}$ and $\Psi=\{\psi_{pq} : p\le q\}$, indexed over the 2-element chain $1<2$, such that the nonidentity maps $\varphi_{12},\psi_{12}\: A_1\to A_2$ are defined by
\[
a\mapsto\varphi_{12}(a) = 1^{\m A_2}\quad\text{and}\quad a\mapsto\psi_{12}(a) = 0,
\]
with $0$ a fixed element in $A_2$ such that $0 < 1^{\m A_2}$. Then $\pair{\Phi,\Psi}$ satisfies the conditions of  Theorem~\ref{thm:construction:meta}, and hence we obtain a residuated poset $\m S = \m A_1\uplus\m A_2$ (Fig.~\ref{fig:residuatedposetsum}).

\begin{figure}[ht]
\centering
\begin{tikzpicture}[every label/.append style={font=\scriptsize}]
\node at (0,3)[n]{\scriptsize$\m A_1$};
\node at (5,3)[n]{\scriptsize$\m A_2$};

\draw (5,0)..controls(2,2)and(2,4)..(5,6)..controls(8,4)and(8,2)..(5,0)
(-1.4,4)--(-.8,3.5)--(0,4)--(.8,3.5)--
(1.4,4)..controls(2,3.5)and(2,2.5)..(1.4,2)
(-1.4,4)..controls(-2,3.5)and(-2,2.5)..(-1.4,2)--
(-.8,2.5)--(0,2)--(.8,2.5)--(1.4,2)--(5,1)(0,2)--(5,1)(-1.4,2)--(5,1)
(1.4,4)--(5,5)(0,4)--(5,5)(-1.4,4)--(5,5)
(5,1) node(0p)[label=right:$0$]{}
(5,5) node(1p)[label=-45:$1^{\m A_2}$]{};
\end{tikzpicture}
\caption{Example of a P\l{}onka sum $\m A_1\uplus\m A_2$ of two residuated posets.}
\label{fig:residuatedposetsum}
\end{figure}
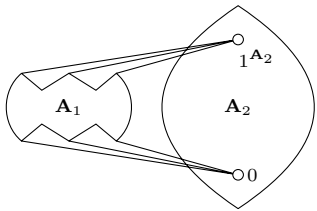

\paragraph{\bfseries Doubly chopped lattices.}
A \emph{doubly chopped lattice} (see \cite{GratzerSchmidt95}, \cite{GratzerBook}) is a poset in which every pair of elements with an upper bound has a join and every pair of elements with a lower bound has a meet. Such a poset will become a lattice if a new top and bottom element is added to the poset, and conversely, every doubly chopped lattice is obtained from a bounded lattice by removing the bounds.

If $\m A$ is a residuated poset that is a doubly chopped lattice and ${\m B}$ is a residuated lattice then the P\l{}onka sum $\m S = \m A\uplus\m B$ is a lattice under the following operations:
{\small
\begin{align*}
a\vee^{\m S} b &= \begin{cases}
  a\vee^{{\m A}} b& \text{ if } a,b\in A\text{ have an upper bound in $A$}\\
  1^\m B& \text{ if } a,b\in A\text{ have no upper bound in $A$}\\
  a\vee^{{\m B}} b& \text{ if } a,b\in B \\
  a &\text{ if } a\in A, \;b\in B \text{ with } b\le 0 \\
  b\vee^{{\m B}} 1^{{\m B}} &\text{ if } a\in A, \;b\in B \text{ with } b\not\le 0,
\end{cases} \\
a\wedge^{\m S} b &= \begin{cases}
  a\wedge^{{\m A}} b&\text{ if } a,b\in A \text{ have a lower bound in $A$}\\
  0&\text{ if } a,b\in A \text{ have no lower bound in $A$}\\
  a\wedge^{{\m B}} b& \text{ if } a,b\in B \\
  a &\text{ if } a\in A, \;b\in B \text{ with } 1^{{\m B}}\le b \\
  b\wedge^{{\m B}} 0 &\text{ if } a\in A, \;b\in B \text{ with } 1^{{\m B}}\not\le b.
\end{cases}
\end{align*}
}

It is easily checked that $a,b\le a\vee^{\m S} b$. Moreover, consider $a,b\le^{\m S} c$, for some $c\in A\cup B $, $a\in A$, and $b\in B$ (the other cases are trivial). If $c\in A$ then $b\le^{\m S} c$ implies $b\le^{\m B} 0$, then $a\vee^{\m S}b = a\le c$. Alternatively, if $c\in B$ then $a\le^\m S c$ implies $1^{\m B}\le c$, then $a\vee^{\m S}b =1^{\m B}\vee^{\m S}b\le c $. The case of $\wedge^{\m S} $ can be checked analogously.

\paragraph{\bfseries The relation algebra $\mathcal P(\mathbb Z_2)$.}
For any monoid $\m M =\pair{M,\cdot,e}$ the \emph{complex algebra} $\mathcal P(\m M)$ is the residuated lattice $\pair{\mathcal P(M),\cap,\cup,\cdot,\ld,\rd,\{e\}}$, where for all $X,Y\subseteq M$, $X\cdot Y=\{xy\: x\in X, y\in Y\}$, $X\ld Y=\{z\in M\: X\cdot\{z\}\subseteq Y\}$ and $X\rd Y=\{z\in M\: \{z\}\cdot Y\subseteq X\}$. The P\l{}onka sum with two components presented above (Fig.~\ref{fig:residuatedposetsum}) encompasses the example of the 4-element relation algebra $\mathcal P(\mathbb Z_2)$ given by the complex algebra of the 2-element group (Fig.~\ref{fig:a2}). Note that this relation algebra is not locally integral~\cite{GFJiLo23}, but it is balanced and satisfies the identities~\ref{H4}--\ref{H6}, hence the P\l{}onka sum decomposition can be applied to all members of the variety of relation algebras generated by $\mathcal P(\mathbb Z_2)$.

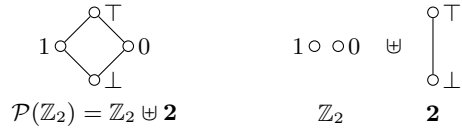
\begin{figure}[ht]
\centering
\begin{tikzpicture}[baseline=0pt]%, yscale=.8, xscale=1.2]
\node at (0,-1)[n]{$\mathcal P(\mathbb Z_2)=\mathbb Z_2\uplus\mathbf 2$};
\node(3) at (0,2)[label=right:$\top$]{};
\node(2) at (1,1)[label=right:$0$]{} edge (3);
\node(1) at (-1,1)[label=left:$1$]{} edge (3);
\node(0) at (0,0)[label=right:$\bot$]{} edge (1) edge (2);
\end{tikzpicture}
\qquad \qquad
\begin{tikzpicture}[baseline=0pt]%, yscale=.8, xscale=1.2]
\node at (-2,-1)[n]{$\mathbb Z_2$};
\node at (1,-1)[n]{$\m2$};
\node(2) at (-1.8,1)[label=right:$0$]{};
\node(1) at (-2.5,1)[label=left:$1$]{};
\node(1) at (-.15,1)[n]{$\uplus$};
\node(3) at (1,2)[label=right:$\top$]{};
\node(0) at (1,0)[label=right:$\bot$]{} edge (3);
\end{tikzpicture}

    \caption{A 4-element relation algebra obtained from the P\l{}onka sum of the 2-element group and the 2-element Boolean algebra.}
    \label{fig:a2}
\end{figure}

\section{Conclusion}
We have generalized the concept of P\l{}onka sum to multiple partition functions, which enables us to cover a larger class of algebras. We further characterized when a pair of directed systems of monotone maps can construct a poset from a family of posets. In that way we are able to recover the partial order with respect to which a given tuple of operations is residuated. This allows the reconstruction of balanced residuated posets from components that have the identity element as their unique positive idempotent. In particular the variety of relation algebras generated by $\mathcal P(\mathbb Z_2)$ contains balanced residuated posets (as reducts) that satisfy~\ref{H4}--\ref{H6}, and hence P\l{}onka sums of directed systems of metamorphisms provide a description of the structure of these algebras in terms of simpler components.

\section*{Acknowledgments}

We thank the anonymous referees for their detailed comments and observations, which allowed to improve the quality of the paper.
S. Bonzio acknowledges the support by the Italian Ministry of Education, University and Research through the projects PRIN 2022 DeKLA (``Developing Kleene Logics and their Applications'', code: 2022SM4XC8) and PRIN Pnrr project Qm4Np (``Quantum Models for Logic, Computation and Natural Processes'', code: P2022A52CR). He also acknowledges the Fondazione di Sardegna for the support received by the project MAPS (grant number F73C23001550007). Finally, he gratefully acknowledges also the support of the INDAM GNSAGA (Gruppo Nazionale per le Strutture Algebriche, Geometriche e loro Applicazioni). The work of A. P\v{r}enosil was funded by the grant 2021 BP 00212 of the grant agency AGAUR of the Generalitat de Catalunya. S. Bonzio and A. P\v{r}enosil also acknowledge the support of the MOSAIC project (H2020-MSCA-RISE-2020 Project 101007627), which funded their visits to Chapman University.

\end{document}